\documentclass{article}
\usepackage{amsfonts}
\usepackage{harvard}
\usepackage{amsmath}

\newtheorem{theorem}{Theorem}

\newtheorem{definition}[theorem]{Definition}

\newtheorem{remark}[theorem]{Remark}

\newenvironment{proof}[1][Proof]{\noindent\textbf{#1.} }{\ \rule{0.5em}{0.5em}}
\numberwithin{theorem}{section}
\numberwithin{equation}{section}
\input{tcilatex}

\begin{document}

\title{Jet Riemann-Lagrange Geometry Applied to Evolution DEs Systems from
Economy}
\author{Mircea Neagu}
\date{}
\maketitle

\begin{abstract}
The aim of this paper is to construct a natural Riemann-Lagrange
differential geometry on 1-jet spaces, in the sense of nonlinear
connections, generalized Cartan connections, d-torsions, d-curvatures, jet
electromagnetic fields and jet Yang-Mills energies, starting from some given
non-linear evolution DEs systems modelling economic phenomena, like the
Kaldor model of the bussines cycle or the Tobin-Benhabib-Miyao model
regarding the role of money on economic growth.
\end{abstract}



\section{Historical aspects}

\hspace{4mm} According to Olver and Udri\c{s}te opinions expressed in [9],
[13] and in private discussions, a lot of applicative problems coming from
Physics [13], Biology [6] or Economics [14] can be modelled on 1-jet spaces $%
J^{1}(T,M)$, where $T^{p}$ is a smooth \textit{"multi-time"} manifold of
dimension $p$ and $M^{n}$ is a smooth \textit{"spatial"} manifold of
dimension $n$. In a such context, a lot of authors (for example, Asanov [1],
Saunders [11], Vondra [15] and many others) studied the \textit{%
contravariant differential geometry} of 1-jet spaces. Proceeding with the
geometrical studies of Asanov, the author of this paper has recently
developed the \textit{Riemann-Lagrange geometry of 1-jet spaces} [5], which
is a natural extension on 1-jet spaces of the well known \textit{Lagrange
geometry of the tangent bundle} due to Miron and Anastasiei [4].

It is important to note that the Riemann-Lagrange geometry of 1-jet spaces
contains many fruitful ideas for the geometrical interpretation of the
solutions of a given DEs or PDEs system [7]. For instance, Udri\c{s}te
proved in [13] that the orbits of a given vector field may be regarded as
horizontal geodesics in a suitable Riemann-Lagrange geometrical structure,
solving in this way an old open problem suggested by Poincar\'{e} [10]: 
\textit{Find the geometric structure which transforms the field lines of a
given vector field into geodesics.}

In the sequel, we present the main geometrical ideas used by Udri\c{s}te in
order to solve the open problem of Poincar\'{e}. For more details, the
reader is invited to consult the works [13] and [14].

In this direction, let us consider a Riemannian manifold $(M^{n},\varphi
_{ij}(x))$ and let us fix an arbitrary vector field $X=(X^{i}(x))$ on $M$.
Obviously, the vector field $X$ produces the first order DEs system (\textit{%
dynamical system})%
\begin{equation}
\frac{dx^{i}}{dt}=X^{i}(x(t)),\text{ }\forall \text{ }i=\overline{1,n}.
\label{ODEs}
\end{equation}

Differentiating the first order DEs system (\ref{ODEs}) and making a
convenient arranging of the terms involved, via the vector field $X$, the
Riemannian metric $\varphi _{ij}$ and its Christoffel symbols $\gamma
_{jk}^{i}$, Udri\c{s}te constructs a second order prolongation (\textit{%
single-time geometric dynamical system}) having the form%
\begin{equation}
\frac{d^{2}x^{i}}{dt^{2}}+\gamma _{jk}^{i}\frac{dx^{j}}{dt}\frac{dx^{k}}{dt}%
=F_{j}^{i}\frac{dx^{j}}{dt}+\varphi ^{ih}\varphi _{kj}X^{j}\nabla _{h}X^{k},%
\text{ }\forall \text{ }i=\overline{1,n},  \label{SODEs}
\end{equation}%
where $\nabla $ is the Levy-Civita connection of the Riemannian manifold $%
(M,\varphi )$ and 
\begin{equation*}
F_{j}^{i}=\nabla _{j}X^{i}-\varphi ^{ih}\varphi _{kj}\nabla _{h}X^{k}.
\end{equation*}

\begin{remark}
Note that any solution of class $C^{2}$ of the first order DEs system (\ref%
{ODEs}) is also a solution for the second order DEs system (\ref{SODEs}).
Conversely, this statement is not true.
\end{remark}

The second order DEs system (\ref{SODEs}) is important because it is
equivalent with the Euler-Lagrange equations of that so-called the \textit{%
least squares Lagrangian function }%
\begin{equation*}
LS:TM\rightarrow \mathbb{R}_{+},\mathit{\ }
\end{equation*}%
given by%
\begin{equation}
LS(x,y)=\frac{1}{2}\varphi _{ij}(x)\left[ y^{i}-X^{i}(x)\right] \left[
y^{j}-X^{j}(x)\right] .  \label{LS1}
\end{equation}

It is obvious now that the field lines of class $C^{2}$ of the vector field $%
X$ are the \textit{global minimum points} of the \textit{least squares
energy action} attached to $LS$, so these field lines are solutions of the
Euler-Lagrange equations produced by $LS.$ Because the Euler-Lagrange
equations of $LS$ are exactly the equations (\ref{SODEs}), Udri\c{s}te
asserts that the solutions of class $C^{2}$ of the first order DEs system (%
\ref{ODEs}) are \textit{horizontal geodesics} on the \textit{%
Riemann-Lagrange manifold}%
\begin{equation*}
(\mathbb{R}\times M,1+\varphi ,N(_{1}^{i})_{j}=\gamma
_{jk}^{i}y^{k}-F_{j}^{i}).
\end{equation*}

\begin{remark}
The author of this paper believe that the preceding least squares
variational method for the geometrical study of the DEs system (\ref{ODEs})
can be reduced to a natural extension of the following well known and simple
idea coming from linear algebra: \textbf{In any Euclidian vector space }$%
(V,\left\langle ,\right\rangle )$\textbf{\ the following equivalence holds
good: }%
\begin{equation*}
v=0_{V}\Leftrightarrow ||v||=0.
\end{equation*}
\end{remark}

Using as a pattern the geometrical Udri\c{s}te's ideas, in what follows we
expose the main geometrical results on 1-jet spaces that, in our opinion,
characterize a given first order non-linear DEs system regarded as an
ordinary differential system on an 1-jet space $J^{1}(T,M)$, where $T\subset 
\mathbb{R}$.

\section{Jet Riemann-Lagrange geometry produced by a non-linear DEs system
of order one}

\hspace{4mm} Let $T=[a,b]\subset \mathbb{R}$ be a compact interval of the
set of real numbers and let us consider the jet fibre bundle of order one%
\begin{equation*}
J^{1}(T,\mathbb{R}^{n})\rightarrow T\times \mathbb{R}^{n},\text{ }n\geq 2,
\end{equation*}%
whose local coordinates $(t,x^{i},x_{1}^{i}),$ $i=\overline{1,n},$ transform
by the rules%
\begin{equation*}
\widetilde{t}=\widetilde{t}(t),\text{ }\widetilde{x}^{i}=\widetilde{x}%
^{i}(x^{j}),\text{ }\widetilde{x}_{1}^{i}=\frac{\partial \widetilde{x}^{i}}{%
\partial x^{j}}\frac{dt}{d\widetilde{t}}\cdot x_{1}^{j}.
\end{equation*}

\begin{remark}
From a physical point of view, the coordinate $t$ has the physical meaning
of \textbf{relativistic time}, the coordinates $(x^{i})_{i=\overline{1,n}}$
represent \textbf{spatial coordinates} and the coordinates $(x_{1}^{i})_{i=%
\overline{1,n}}$ have the physical meaning of \textbf{relativistic velocities%
}.
\end{remark}

Let us consider that $X=\left( X_{(1)}^{(i)}(x^{k})\right) $ is an arbitrary
d-tensor field on the 1-jet space $J^{1}(T,\mathbb{R}^{n})$, whose local
components transform by the rules%
\begin{equation*}
\widetilde{X}_{(1)}^{(i)}=\frac{\partial \widetilde{x}^{i}}{\partial x^{j}}%
\frac{dt}{d\widetilde{t}}\cdot X_{(1)}^{(j)}.
\end{equation*}%
Obviously, the d-tensor field $X$ produces the jet DEs system of order one (%
\textit{jet dynamical system})%
\begin{equation}
x_{1}^{i}=X_{(1)}^{(i)}(x^{k}(t)),\text{ }\forall \text{ }i=\overline{1,n},
\label{DEs1}
\end{equation}%
where $c(t)=(x^{i}(t))$ is an unknown curve on $\mathbb{R}^{n}$ (i. e., a
jet field line of the d-tensor field $X$) and we used the notation%
\begin{equation*}
x_{1}^{i}\overset{not}{=}\dot{x}^{i}=\frac{dx^{i}}{dt},\text{ }\forall \text{
}i=\overline{1,n}.
\end{equation*}

Supposing now that we have the Euclidian structures $(T,1)$ and $(\mathbb{R}%
^{n},\delta _{ij})$, where $\delta _{ij}$ are the Kronecker symbols, then
the jet first order DEs system (\ref{DEs1}) automatically produces the 
\textit{jet least squares Lagrangian function} 
\begin{equation*}
JLS:J^{1}(T,\mathbb{R}^{n})\rightarrow \mathbb{R}_{+},
\end{equation*}%
expressed by%
\begin{equation}
JLS(x^{k},x_{1}^{k})=\sum_{i=1}^{n}\left[ x_{1}^{i}-X_{(1)}^{(i)}(x)\right]
^{2},  \label{JetLS}
\end{equation}%
where $x=(x^{k})_{k=\overline{1,n}}.$ Of course, the \textit{global minimum
points} of the \textit{jet least squares energy action}%
\begin{equation*}
\mathbb{E}(c(t))=\int_{a}^{b}JLS(x^{k}(t),\dot{x}^{k}(t))dt
\end{equation*}%
are exactly the solutions of class $C^{2}$ of the jet first order DEs system
(\ref{DEs1}). In other words, the solutions of class $C^{2}$ of the jet DEs
system of order one (\ref{DEs1}) verify the second order Euler-Lagrange
equations produced by $JLS$ (\textit{\ jet geometric dinamics}).

\begin{remark}
Because a Riemann-Lagrange geometry on $J^{1}(T,\mathbb{R}^{n})$ produced by
the jet least squares Lagrangian function $JLS$, via its second order
Euler-Lagrange equations, in the sense of non-linear connection, generalized
Cartan connection, d-torsions, d-curvatures, jet electromagnetic field and
jet Yang-Mills energy, is now completely done in the papers [5], [6] and
[7], it follows that we may regard $JLS$ as a natural geometrical substitut
on $J^{1}(T,\mathbb{R}^{n})$ for the jet first order DEs system (\ref{DEs1}).
\end{remark}

In this context, we introduce the following notion:

\begin{definition}
Any geometrical object on $J^{1}(T,\mathbb{R}^{n})$, which is produced by
the jet least squares Lagrangian function $JLS$, via its second order
Euler-Lagrange equations, is called \textbf{geometrical object produced by
the jet first order DEs system (\ref{DEs1})}.
\end{definition}

In order to expose the main jet Riemann-Lagrange geometrical objects that
characterize the jet first order DEs system (\ref{DEs1}), we use the
following matriceal Jacobian notation:%
\begin{equation*}
J\left( X_{(1)}\right) =\left( \frac{\partial X_{(1)}^{(i)}}{\partial x^{j}}%
\right) _{i,j=\overline{1,n}}.
\end{equation*}%
In this context, the following geometrical result, which is proved in [6]
and, for more general cases, in [5] and [7], holds good.

\begin{theorem}
\label{MainTh}(i) The \textbf{canonical non-linear connection on }$J^{1}(T,%
\mathbb{R}^{n})$\textbf{\ produced by the jet first order DEs system (\ref%
{DEs1})} has the local components%
\begin{equation*}
\Gamma =\left( 0,N_{(1)j}^{(i)}\right) ,
\end{equation*}%
where $N_{(1)j}^{(i)}$ are the entries of the matrix%
\begin{equation*}
N_{(1)}=\left( N_{(1)j}^{(i)}\right) _{i,j=\overline{1,n}}=-\frac{1}{2}\left[
J\left( X_{(1)}\right) -\text{ }^{T}J\left( X_{(1)}\right) \right] .
\end{equation*}

(ii) All adapted components of the \textbf{canonical generalized Cartan
connection }$C\Gamma $\textbf{\ produced by the jet first order DEs system (%
\ref{DEs1})} vanish.

(iii) The effective adapted components $R_{(1)jk}^{(i)}$ of the \textbf{%
torsion} d-tensor \textbf{T} of the canonical generalized Cartan connection $%
C\Gamma $ \textbf{produced by the jet first order DEs system (\ref{DEs1})}
are the entries of the matrices%
\begin{equation*}
R_{(1)k}=\frac{\partial }{\partial x^{k}}\left[ N_{(1)}\right] ,\text{ }%
\forall \text{ }k=\overline{1,n},
\end{equation*}%
where%
\begin{equation*}
R_{(1)k}=\left( R_{(1)jk}^{(i)}\right) _{i,j=\overline{1,n}},\text{ }\forall 
\text{ }k=\overline{1,n}.
\end{equation*}

(iv) All adapted components of the \textbf{curvature} d-tensor \textbf{R} of
the canonical generalized Cartan connection $C\Gamma $ \textbf{produced by
the jet first order DEs system (\ref{DEs1})} vanish.

(v) The \textbf{geometric electromagnetic distinguished 2-form produced by
the jet first order DEs system (\ref{DEs1})} has the expression%
\begin{equation*}
F=F_{(i)j}^{(1)}\delta x_{1}^{i}\wedge dx^{j},
\end{equation*}%
where%
\begin{equation*}
\delta x_{1}^{i}=dx_{1}^{i}+N_{(1)k}^{(i)}dx^{k},\text{ }\forall \text{ }i=%
\overline{1,n},
\end{equation*}%
and the adapted components $F_{(i)j}^{(1)}$ are the entries of the matrix%
\begin{equation*}
F^{(1)}=\left( F_{(i)j}^{(1)}\right) _{i,j=\overline{1,n}}=-N_{(1)}.
\end{equation*}

(vi) The adapted components $F_{(i)j}^{(1)}$ of the geometric
electromagnetic d-form $F$ produced by the jet first order DEs system (\ref%
{DEs1})\textbf{\ }verify the \textbf{generalized Maxwell equations}%
\begin{equation*}
\sum_{\{i,j,k\}}F_{(i)j||k}^{(1)}=0,
\end{equation*}%
where $\sum_{\{i,j,k\}}$ represents a cyclic sum and%
\begin{equation*}
F_{(i)j||k}^{(1)}=\frac{\partial F_{(i)j}^{(1)}}{\partial x^{k}}
\end{equation*}%
means the horizontal local covariant derivative produced by the Berwald
connection $B\Gamma _{0}$ on $J^{1}(T,\mathbb{R}^{n}).$ For more details,
please consult [5].

(vii) The \textbf{geometric jet Yang-Mills energy produced by the jet first
order DEs system (\ref{DEs1})} is given by the formula%
\begin{equation*}
EYM(x)=\frac{1}{2}\cdot Trace\left[ F^{(1)}\cdot \text{ }^{T}F^{(1)}\right] ,
\end{equation*}%
where%
\begin{equation*}
F^{(1)}=\left( F_{(i)j}^{(1)}\right) _{i.j=\overline{1,n}}.
\end{equation*}
\end{theorem}

In the next Sections, we apply the above jet Riemann-Lagrange geometrical
results to certain evolution equations that govern diverse phenomena in
Economy, extending in this way the geometrical studies initiated by Udri\c{s}%
te, Ferrara and Opri\c{s} in the book [14].

\section{Jet Riemann-Lagrange geometry for Kaldor non-linear cyclical model
in business}

\hspace{4mm} The \textit{national revenue} $Y(t)$ and the \textit{capital
stock} $K(t)$, where $t\in \lbrack a,b]$, are the state variables of the
Kaldor non-linear model of the business cycle. The kinetic Kaldor model of a
commercial cycle belongs to the category of business cycles described by the 
\textit{Kaldor flow} (for more details, please see [3], [14])%
\begin{equation}
\left\{ 
\begin{array}{l}
\dfrac{dY}{dt}=s\left[ I(Y,K)-S(Y,K)\right] \medskip \\ 
\dfrac{dK}{dt}=I(Y,K)-qK,%
\end{array}%
\right.  \label{Kaldor}
\end{equation}%
where

\begin{itemize}
\item $I=I(Y,K)$ is a given differentiable \textit{investment function},
which verifies some \textit{economic-mathematical Kaldor conditions};

\item $S=S(Y,K)$ is a given differentiable \textit{saving function}, which
verifies some \textit{economic-mathematical Kaldor conditions};

\item $s>0$ is an adjustement constant parameter which measures the reaction
of the model with respect to the difference between the investment function
and the saving function;

\item $q\in (0,1)$ is a constant representing the depreciation coefficient
of capital.
\end{itemize}

\begin{remark}
Details upon the \textbf{Kaldor economic-mathematical conditions} imposed to
the given functions $I$ and $S$ find in [3] and [14]. From the point of view
of the jet Riemann-Lagrange geometry produced by the Kaldor evolution
equations, geometry that we will describe in the sequel, the
economic-mathematical hypotheses of Kaldor can be neglected because all our
geometrical informations are concentrated in the Kaldor flow (\ref{Kaldor}).
\end{remark}

The Riemann-Lagrange geometrical behavior on the 1-jet space $J^{1}(T,%
\mathbb{R}^{2})$ of the Kaldor economic evolution model is described in the
following result:

\begin{theorem}
(i) The \textbf{canonical non-linear connection on }$J^{1}(T,\mathbb{R}^{2})$%
\textbf{\ produced by the Kaldor flow (\ref{Kaldor})} has the local
components%
\begin{equation*}
\hat{\Gamma}=\left( 0,\hat{N}_{(1)j}^{(i)}\right) ,
\end{equation*}%
where, if $I_{Y},$ $I_{K}$ and $S_{K}$ are the partial derivatives of the
functions $I$ and $S$, then $\hat{N}_{(1)j}^{(i)}$ are the entries of the
matrix%
\begin{equation*}
\hat{N}_{(1)}=\left( 
\begin{array}{cc}
0 & \dfrac{1}{2}[I_{Y}-s(I_{K}-S_{K})]\medskip \\ 
-\dfrac{1}{2}[I_{Y}-s(I_{K}-S_{K})] & 0%
\end{array}%
\right) .
\end{equation*}

(ii) All adapted components of the \textbf{canonical generalized Cartan
connection }$C\hat{\Gamma}$\textbf{\ produced by the Kaldor flow (\ref%
{Kaldor})} vanish.

(iii) All adapted components of the \textbf{torsion} d-tensor \textbf{\^{T}}
of the canonical generalized Cartan connection $C\hat{\Gamma}$ \textbf{%
produced by the Kaldor flow (\ref{Kaldor})} are zero, except%
\begin{equation*}
\begin{array}{l}
\hat{R}_{(1)21}^{(1)}=-\hat{R}_{(1)11}^{(2)}=\dfrac{1}{2}%
[I_{YY}-s(I_{YK}-S_{YK})],\medskip \\ 
\hat{R}_{(1)22}^{(1)}=-\hat{R}_{(1)12}^{(2)}=\dfrac{1}{2}%
[I_{YK}-s(I_{KK}-S_{KK})],%
\end{array}%
\end{equation*}%
where $I_{YY}$, $I_{YK}$, $I_{KK}$, $S_{YK}$ and $S_{KK}$ are the second
partial derivatives of the functions $I$ and $S$.

(iv) All adapted components of the \textbf{curvature} d-tensor \textbf{\^{R}}
of the canonical generalized Cartan connection $C\hat{\Gamma}$ \textbf{%
produced by the Kaldor flow (\ref{Kaldor})} vanish.

(v) The \textbf{geometric electromagnetic distinguished 2-form produced by
the Kaldor flow (\ref{Kaldor})} has the expression%
\begin{equation*}
\hat{F}=\hat{F}_{(i)j}^{(1)}\delta x_{1}^{i}\wedge dx^{j},
\end{equation*}%
where%
\begin{equation*}
\delta x_{1}^{i}=dx_{1}^{i}+\hat{N}_{(1)k}^{(i)}dx^{k},\text{ }\forall \text{
}i=\overline{1,2},
\end{equation*}%
and the adapted components $\hat{F}_{(i)j}^{(1)}$ are the entries of the
matrix%
\begin{equation*}
\hat{F}^{(1)}=-\hat{N}_{(1)}=\left( 
\begin{array}{cc}
0 & -\dfrac{1}{2}[I_{Y}-s(I_{K}-S_{K})]\medskip \\ 
\dfrac{1}{2}[I_{Y}-s(I_{K}-S_{K})] & 0%
\end{array}%
\right) .
\end{equation*}

(vi) The \textbf{economic geometric Yang-Mills energy produced by the Kaldor
flow (\ref{Kaldor})} is given by the formula%
\begin{equation*}
EYM^{\text{Kaldor}}(Y,K)=\frac{1}{4}\left[ I_{Y}-s(I_{K}-S_{K})\right] ^{2}.
\end{equation*}
\end{theorem}

\begin{proof}
Let us regard the Kaldor flow (\ref{Kaldor}) as a particular case of the jet
first order DEs system (\ref{DEs1}) on the 1-jet space $J^{1}(T,\mathbb{R}%
^{2})$, taking 
\begin{equation*}
n=2,\text{ }x^{1}=Y,\text{ }x^{2}=K
\end{equation*}%
and putting%
\begin{equation*}
X_{(1)}^{(1)}(x^{1},x^{2})=s\left[ I(x^{1},x^{2})-S(x^{1},x^{2})\right] 
\text{ and }X_{(1)}^{(2)}(x^{1},x^{2})=I(x^{1},x^{2})-qx^{2}.
\end{equation*}%
Now, taking into account that we have the Jacobian matrix%
\begin{eqnarray*}
J\left( X_{(1)}\right) &=&\left( 
\begin{array}{cc}
s\left[ I_{x^{1}}-S_{x^{1}}\right] & s\left[ I_{x^{2}}-S_{x^{2}}\right] \\ 
I_{x^{1}} & I_{x^{2}}-q%
\end{array}%
\right) \\
&=&\left( 
\begin{array}{cc}
s\left[ I_{Y}-S_{Y}\right] & s\left[ I_{K}-S_{K}\right] \\ 
I_{Y} & I_{K}-q%
\end{array}%
\right) ,
\end{eqnarray*}%
and using the Theorem \ref{MainTh}, we obtain what we were looking for.
\end{proof}

\begin{remark}[Open problem]
The \textbf{Yang-Mills economic e\-ner\-ge\-ti\-cal \linebreak curves of
constant level produced by the Kaldor flow (\ref{Kaldor})}, which are
different by the empty set, are the curves in the plane $YOK$ having the
implicit equations%
\begin{equation*}
\mathcal{C}_{C}:\left[ I_{Y}-s(I_{K}-S_{K})\right] ^{2}=4C,
\end{equation*}%
where $C\geq 0.$ Is it possible as the shapes of the plane curves $\mathcal{C%
}_{C}$ to offer economic interpretations for economists?
\end{remark}

\section{Jet Riemann-Lagrange geometry for Tobin-Benhabib-Miyao economic
evolution model}

\hspace{4mm}The Tobin mathematical model [12] regarding the role of money on
economic growth was extended by Benhabib and Miyao [2] by incorporating the
role of some expectation constant parameters. Thus, the Tobin-Benhabib-Miyao
(TBM) economic model relies on the variables $k(t)=$ \textit{the capital
labor ratio}, $m(t)=$ \textit{the money stock per head}, $q(t)=$ \textit{the
expected rate of inflation}, whose evolution in time is given by the \textit{%
TBM flow} [14]%
\begin{equation}
\left\{ 
\begin{array}{l}
\dfrac{dk}{dt}=sf(k(t))-(1-s)[\theta -q(t)]m(t)-nk(t)\medskip \\ 
\dfrac{dm}{dt}=m(t)\left\{ \theta -n-q(t)-\varepsilon \lbrack
m(t)-l(k(t),q(t))]\right\} \medskip \\ 
\dfrac{dq}{dt}=\mu \varepsilon \lbrack m(t)-l(k(t),q(t))],%
\end{array}%
\right.  \label{TBM}
\end{equation}%
where the $f(k)$ and $l(k,q)$ are some given differentiable real functions
and $s$, $\theta $, $n$, $\mu $, $\varepsilon $ are expectation parameters: $%
s=$ \textit{saving ratio}, $\theta =$ \textit{rate of money expansion}, $n=$ 
\textit{population growth rate}, $\mu =$ \textit{speed of adjustement of
expectations}, $\varepsilon =$ \textit{speed of adjustement of price level}.

\begin{remark}
From the point of view of economists, the \textbf{actual rate of inflation}
in the TBM economic model is given by the formula [14]%
\begin{equation*}
\overline{p}(t)=\varepsilon \lbrack m(t)-l(k(t),q(t))]+q(t).
\end{equation*}
\end{remark}

In what follows, we apply our jet Riemann-Lagrange geometrical results to
the TBM flow (\ref{TBM}). In this context, we obtain:

\begin{theorem}
(i) The \textbf{canonical non-linear connection on }$J^{1}(T,\mathbb{R}^{3})$%
\textbf{\ produced by the TBM flow (\ref{TBM})} has the local components%
\begin{equation*}
\check{\Gamma}=\left( 0,\check{N}_{(1)j}^{(i)}\right) ,
\end{equation*}%
where, if $l_{k}$ and $l_{q}$ are the partial derivatives of the function $l$%
, then $\check{N}_{(1)j}^{(i)}$ are the entries of the matrix%
\begin{equation*}
\check{N}_{(1)}=-\frac{1}{2}\left( 
\begin{array}{ccc}
0 & 
\begin{array}{c}
-(1-s)(\theta -q)- \\ 
-\varepsilon ml_{k}%
\end{array}
& (1-s)m+\mu \varepsilon l_{k}\medskip \\ 
\begin{array}{c}
(1-s)(\theta -q)+ \\ 
+\varepsilon ml_{k}%
\end{array}
& 0 & -m+\varepsilon ml_{q}-\mu \varepsilon \medskip \\ 
-(1-s)m-\mu \varepsilon l_{k} & m-\varepsilon ml_{q}+\mu \varepsilon & 0%
\end{array}%
\right) .
\end{equation*}

(ii) All adapted components of the \textbf{canonical generalized Cartan
connection }$C\check{\Gamma}$\textbf{\ produced by the TBM flow (\ref{TBM})}
vanish.

(iii) The effective adapted components of the \textbf{torsion} d-tensor 
\textbf{\v{T}} of the canonical generalized Cartan connection $C\check{\Gamma%
}$ \textbf{produced by the TBM flow (\ref{TBM})} are the entries of the
matrices 
\begin{equation*}
\check{R}_{(1)1}=-\frac{1}{2}\left( 
\begin{array}{ccc}
0 & -\varepsilon ml_{kk} & \mu \varepsilon l_{kk} \\ 
\varepsilon ml_{kk} & 0 & \varepsilon ml_{kq} \\ 
-\mu \varepsilon l_{kk} & -\varepsilon ml_{kq} & 0%
\end{array}%
\right) ,
\end{equation*}%
\begin{equation*}
\check{R}_{(1)2}=-\frac{1}{2}\left( 
\begin{array}{ccc}
0 & -\varepsilon l_{k} & 1-s \\ 
\varepsilon l_{k} & 0 & -1+\varepsilon l_{q} \\ 
-1+s & 1-\varepsilon l_{q} & 0%
\end{array}%
\right)
\end{equation*}%
and%
\begin{equation*}
\check{R}_{(1)3}=-\frac{1}{2}\left( 
\begin{array}{ccc}
0 & 1-s-\varepsilon ml_{kq} & \mu \varepsilon l_{kq} \\ 
-1+s+\varepsilon ml_{kq} & 0 & \varepsilon ml_{qq} \\ 
-\mu \varepsilon l_{kq} & -\varepsilon ml_{qq} & 0%
\end{array}%
\right) ,
\end{equation*}%
where $l_{kk}$, $l_{kq}$ and $l_{qq}$ are the second partial derivatives of
the function $l$ and%
\begin{equation*}
\check{R}_{(1)k}=\left( \check{R}_{(1)jk}^{(i)}\right) _{i,j=\overline{1,3}},%
\text{ }\forall \text{ }k=\overline{1,3}.
\end{equation*}

(iv) All adapted components of the \textbf{curvature} d-tensor \textbf{\v{R}}
of the canonical generalized Cartan connection $C\check{\Gamma}$ \textbf{%
produced by the TBM flow (\ref{TBM})} vanish.

(v) The \textbf{geometric electromagnetic distinguished 2-form produced by
the TBM flow (\ref{TBM})} has the expression%
\begin{equation*}
\check{F}=\check{F}_{(i)j}^{(1)}\delta x_{1}^{i}\wedge dx^{j},
\end{equation*}%
where%
\begin{equation*}
\delta x_{1}^{i}=dx_{1}^{i}+\check{N}_{(1)k}^{(i)}dx^{k},\text{ }\forall 
\text{ }i=\overline{1,3},
\end{equation*}%
and the adapted components $\check{F}_{(i)j}^{(1)}$ are the entries of the
matrix%
\begin{equation*}
\check{F}^{(1)}=-\check{N}_{(1)}.
\end{equation*}

(vi) The \textbf{economic geometric Yang-Mills energy produced by the TBM
flow (\ref{TBM})} is given by the formula%
\begin{eqnarray*}
EYM^{\text{TBM}}(k,m,q) &=&\frac{1}{4}\left\{ \left[ (1-s)(\theta
-q)+\varepsilon ml_{k}\right] ^{2}+\left[ (1-s)m+\mu \varepsilon l_{k}\right]
^{2}+\right. \\
&&\left. +\left[ m-\varepsilon ml_{q}+\mu \varepsilon \right] ^{2}\right\} .
\end{eqnarray*}
\end{theorem}

\begin{proof}
We regard the TBM flow (\ref{TBM}) as a particular case of the jet first
order DEs system (\ref{DEs1}) on the 1-jet space $J^{1}(T,\mathbb{R}^{3})$,
taking 
\begin{equation*}
n=3,\text{ }x^{1}=k,\text{ }x^{2}=m,\text{ }x^{3}=q
\end{equation*}%
and putting%
\begin{equation*}
X_{(1)}^{(1)}(x^{1},x^{2},x^{3})=sf(x^{1})-(1-s)[\theta -x^{3}]x^{2}-nx^{1},
\end{equation*}%
\begin{equation*}
X_{(1)}^{(2)}(x^{1},x^{2},x^{3})=x^{2}\left\{ \theta -n-x^{3}-\varepsilon
\lbrack x^{2}-l(x^{1},x^{3})]\right\}
\end{equation*}%
and%
\begin{equation*}
X_{(1)}^{(3)}(x^{1},x^{2},x^{3})=\mu \varepsilon \lbrack
x^{2}-l(x^{1},x^{3})].
\end{equation*}%
It follows that we have the Jacobian matrix%
\begin{eqnarray*}
J\left( X_{(1)}\right) &=&\left( 
\begin{array}{ccc}
sf^{\prime }(x^{1})-n & -(1-s)(\theta -x^{3}) & (1-s)x^{2}\medskip \\ 
\varepsilon x^{2}l_{x^{1}} & 
\begin{array}{c}
-2\varepsilon x^{2}+\theta -x^{3}-n+ \\ 
+\varepsilon l(x^{1},x^{3})%
\end{array}
& -x^{2}+\varepsilon x^{2}l_{x^{3}}\medskip \\ 
-\mu \varepsilon l_{x^{1}} & \mu \varepsilon & -\mu \varepsilon l_{x^{3}}%
\end{array}%
\right) \\
&=&\left( 
\begin{array}{ccc}
sf^{\prime }(k)-n & -(1-s)(\theta -q) & (1-s)m\medskip \\ 
\varepsilon ml_{k} & 
\begin{array}{c}
-2\varepsilon m+\theta -q-n+ \\ 
+\varepsilon l(k,q)%
\end{array}
& -m+\varepsilon ml_{q}\medskip \\ 
-\mu \varepsilon l_{k} & \mu \varepsilon & -\mu \varepsilon l_{q}%
\end{array}%
\right) ,
\end{eqnarray*}%
where $f^{\prime }$ is the derivative of the function $f$. In conclusion,
using the Theorem \ref{MainTh}, we find the required result.
\end{proof}

\begin{remark}[Open problem]
The \textbf{Yang-Mills economic energetical \linebreak surfaces of constant
level produced by the TBM flow (\ref{TBM})}, which are different by the
empty set, have in the system of axis $Okmq$ the implicit equations%
\begin{equation*}
\Sigma _{C}:\left[ (1-s)(\theta -q)+\varepsilon ml_{k}\right] ^{2}+\left[
(1-s)m+\mu \varepsilon l_{k}\right] ^{2}+\left[ m-\varepsilon ml_{q}+\mu
\varepsilon \right] ^{2}=4C,
\end{equation*}%
where $C\geq 0.$ Is it possible as the geometry of the surfaces $\Sigma _{C}$
to contains valuable economic informations for economists?
\end{remark}

\textbf{Author's address: }Mircea NEAGU, Str. L\u{a}m\^{a}i\c{t}ei, Nr. 66,
Bl. 93, Sc. G, Ap. 10, Bra\c{s}ov, BV 500371, Romania

\textbf{E-mail}: mirceaneagu73@yahoo.com

\begin{center}
University "Transilvania" of Bra\c{s}ov

Faculty of Mathematics and Informatics
\end{center}


\begin{thebibliography}{99}
\bibitem{Asanov[1]} G. S. Asanov, \textit{Jet Extension of Finslerian Gauge
Approach}, Fortschritte der Physik \textbf{38}, No. \textbf{8} (1990),
571-610.

\bibitem{Benhabib[2]} J. Benhabib, T. Miyao, \textit{Some New Results on the
Dynamics of the Generalized Tobin Model}, International Economic Review 
\textbf{22}, No. \textbf{3} (1981), 589-596.

\bibitem{Gandolfo[3]} G. Gandolfo, \textit{Economic Dynamics},
Springer-Verlag, 1997.

\bibitem{Mir-An[4]} R. Miron, M. Anastasiei, \textit{The Geometry of
Lagrange Spaces: Theory and Applications}, Kluwer Academic Publishers, 1994.

\bibitem{Neagu Carte[5]} M. Neagu, \textit{Riemann-Lagrange Geometry on
1-Jet Spaces}, Matrix Rom, Bucharest, 2005.

\bibitem{Rodica[6]} M. Neagu, I. R. Nicola, \textit{Geometric Dynamics of
Calcium Oscillations ODEs Systems}, Balkan Journal of Geometry and Its
Applications, Vol. \textbf{9}, No. \textbf{2 }(2004), 36-67.

\bibitem{TM[7]} M. Neagu, C. Udri\c{s}te, \textit{From PDEs Systems and
Metrics to Ge\-o\-me\-tric Multi-Time Field Theories}, Seminarul de Mecanic%
\u{a}, Sisteme Dinamice Diferen\c{t}iale, No. \textbf{79} (2001), Timi\c{s}%
oara, Romania.

\bibitem{Obadeanu[8]} V. Ob\u{a}deanu, \textit{Sisteme Dinamice Diferen\c{t}%
iale. Dinamica Materiei Amorfe}, Editura Universit\u{a}\c{t}ii de Vest, Timi%
\c{s}oara, 2006 (in Romanian).

\bibitem{Olver[9]} P. J. Olver, \textit{Applications of Lie Groups to
Differential Equations}, Springer-Verlag, 1986.

\bibitem{Poincare[10]} H. Poincar\'{e}, \textit{Sur les Courbes Definies par
les Equations Diff\'{e}rentielle}, C.R. Acad. Sci., Paris \textbf{90}
(1880), 673-675.

\bibitem{Saunders[11]} D. Saunders, \textit{The Geometry of Jet Bundles},
Cambridge University Press, New York, London, 1989.

\bibitem{Tobin[12]} J. Tobin, \textit{Money and Economic Growth},
Econometrica \textbf{33} (1965), 671-684.

\bibitem{Udr Geom Dyn[13]} C. Udri\c{s}te, \textit{Geometric Dynamics},
Kluwer Academic Publishers, 2000.

\bibitem{Udr Econ Dyn[14]} C. Udri\c{s}te, M. Ferrara, D. Opri\c{s}, \textit{%
Economic Geometric Dynamics}, Geometry Balkan Press, Bucharest, 2004.

\bibitem{Vondra[15]} A. Vondra, \textit{Symmetries of Connections on Fibered
Manifolds}, Archivum Mathematicum, Brno, Tomus \textbf{30} (1994), 97-115.
\end{thebibliography}
\end{document}